\documentclass[12pt]{amsart}
\usepackage{amsfonts,amssymb,latexsym,amsmath, amsxtra, stmaryrd}
\usepackage[all]{xy}
\usepackage{graphicx}
\usepackage{color}
\theoremstyle{plain}
\newtheorem{theorem}{Theorem}[section]

\newtheorem{proposition}[theorem]{Proposition}

\newtheorem*{conjecture*}{Conjecture}
\newtheorem*{challenge*}{Open Problem}

\theoremstyle{definition}
\newtheorem{definition}[theorem]{Definition}
\newtheorem*{definition*}{Definition}

\newtheorem{example}[theorem]{Example}
\theoremstyle{remark}
\newtheorem*{remark}{Remark}

\numberwithin{equation}{section}

\newcommand{\eps}{\epsilon}
\newcommand{\R}{\mathbb R}

\newcommand{\Z}{\mathbb Z}
\newcommand{\C}{\mathbb C}
\def\H{\mathbb H}

\newcommand{\Q}{{\mathbb Q}}

\def\({\left(}
\def\){\right)}

\newcommand{\ol}[1]{\overline{{#1}}}

\newcommand{\abs}[1]{\left|#1\right|}

\newcommand{\Gal}{\mathrm{Gal}}

\newcommand{\bC}{{\mathbb C}}

\newcommand{\bH}{{\mathbb H}}

\newcommand{\bQ}{{\mathbb Q}}
\newcommand{\bR}{{\mathbb R}}

\newcommand{\bZ}{{\mathbb Z}}

\newcommand{\cG}{{\mathcal G}}

\newcommand{\cN}{{\mathcal N}}
\newcommand{\cO}{{\mathcal O}}

\newcommand{\cS}{{\mathcal S}}

\newcommand{\fa}{{\mathfrak a}}

\newcommand{\ff}{{\mathfrak f}}

\newcommand{\fm}{{\mathfrak m}}

\newcommand{\uu}{u}
\newcommand{\cn}{A_n}

\def\k2{\frac{k}{2}}

\begin{document}

\title[renormalization and  quantum modular forms, Part I]{
Renormalization and quantum modular forms, Part I: Maass wave forms}
\author{Yingkun Li}
\address{UCLA Mathematics Department, Box 951555, Los Angeles, CA 90095-1555, USA}
\email{yingkun@math.ucla.edu}

\author{Hieu T. Ngo}
\address{Dept. of Mathematics, University of Michigan, Ann Arbor, MI 48109-1043, USA}
\email{hieu@umich.edu}

\author{Robert C. Rhoades}
\address{Department of Mathematics, Stanford University, Stanford, CA.
94305}
\curraddr{Center for Communications Research,
Princeton, NJ 08540}
\email{rhoades@math.stanford.edu, rob.rhoades@gmail.com}


\thispagestyle{empty} \vspace{.5cm}
\begin{abstract}

Sander Zwegers showed that Ramanujan's mock theta functions are
$q$-hypergeometric series, whose $q$-expansion coefficients are half
of the Fourier coefficients of a non-holomorphic modular form. George
Andrews, Henri Cohen, Freeman Dyson, and Dean Hickerson found a pair
of $q$-hypergeometric series each of which contains half of the
Fourier coefficients of Maass waveform of eigenvalue $1/4$.

This series of papers shows that a $q$-series construction, called
``renormalization'', yields the other half of the Fourier coefficients
from a series which contains half of them.  This construction unifies
examples associated with mock theta functions and examples associated
with Maass waveforms.  Thus confirming a conviction of Freeman Dyson.
This construction is natural in the context of Don Zagier's quantum
modular forms. Detailed discussion of the role quantum modular forms
play in this construction is given.

New examples associated to Maass waveforms are given in Part I.  Part
II contains new examples associated with mock theta functions, and
classical modular forms.  Part II contains an extensive survey of the
``renormalization'' construction.  A large number of examples and open
questions which share similarities to the main examples, but remain
mysterious, are given.
\end{abstract}

\maketitle


\section{Introduction}\label{sec:intro}\label{sec:Intro}
Ramanujan studied the function 
\begin{align}
\sigma(q) :=& 1+ \sum_{n=1}^\infty \frac{q^{\frac{1}{2} n(n+1)}}{(1+q)(1+q^2)\cdots (1+q^n)} \nonumber  \\
 =& \sum_{n=0}^\infty S(n)q^n = 1 + q - q^2 +2 q^3 +  \cdots + 2q^{55} + q^{57} - 2q^{62} + \cdots + 6 q^{1609} + \cdots 
  \label{eqn:sigma}  
\end{align}
in his ``Lost'' Notebook \cite{andrewsEuler}, where $\abs{q} < 1$.  
This series is the generating function for the number of partitions of $n$ into distinct parts with even rank minus the number with odd rank.  The rank of a partition (see \cite{AS-D, dyson44})
is the largest part minus the number of parts. 
George Andrews \cite{andrewsMonthly} 
made the curious conjectures that 
(1) $S(n)$ is zero infinitely often and (2) $\limsup \abs{S(n)} = +\infty$.

%
%

These conjectures were established by Andrews together with Freeman
Dyson and Dean Hickerson \cite{ADH} who showed that $S(n)$ is equal to
the number of inequivalent solutions to
\begin{equation*}\label{eqn:pell}
u^2 - 6v^2 = 24n+1
\end{equation*}
such that $u+3v \equiv \pm 1\pmod{12}$ minus the number of such solutions with $u+3v \equiv \pm 5 \pmod{12}$. 
\begin{remark}
Two solutions $(u,v)$ and $(u',v')$
are equivalent if $(u'+v'\sqrt{6}) = \pm (5 + 2\sqrt{6})^r
(u+v\sqrt{6})$ for some integer $r$.
\end{remark}
It follows from the arithmetic of $\Q(\sqrt{6})$ that $S(n)$ assumes
every integer value infinitely often.  Moreover, Henri Cohen
\cite{cohen} showed that it follows that the integers $S(n)$ are the
``positive'' Fourier coefficients of a Maass wave form $\varphi_0$ of
eigenvalue $1/4$.  Precisely,
$$\varphi_0(z):= 
 \sqrt{y}  \sum_{n\in 24\Z+1} T(n) K_0(2\pi \abs{n} y/24) e^{2\pi i n x/24},$$
satisfies the modular properties
$$\varphi_0(-1/2z) = \ol{\varphi_0(z)} \ \ \ \text{ and } \ \ \ \varphi_0(z+1) = e^{2\pi  i/24} \varphi_0(z),$$ 
where $K_0$ is a $K$-Bessel function and $z := x+iy$ with $x, y\in \R$ such that  $y>0$, and 
$$S(n) = T(24n+1).$$
Thus, the positive Fourier coefficients of the non-holomorphic
modular form $\varphi_0(z)$ are encoded by $\sigma(q)$.

In this paper we show how the series in \eqref{eqn:sigma} reveals both
the positive and negative Fourier coefficients of the Maass
wave form $\varphi_0(z)$.  In Section \ref{sec:Renormalization} a
construction is given to extend a function from $\abs{q}<1$ to a
function on $\abs{q}>1$.  The pair of functions together contain all
of the Fourier coefficients of a (non-holomorphic) modular form.
Additionally, these functions agree at certain ``holes'', i.e. a
subset of $\abs{q} =1$.  This ``leaking'' results in examples of
Zagier's ``quantum modular forms''.

A new example analogous to $\sigma(q)$ is given in Section
\ref{sec:Results}.  A new feature of our work is the relationship
between these ``holes'' and the cuspidality of the underlying modular
form.

This paper deals only with Maass wave forms. Part II of this series
contains examples of the same phenomenon related to Ramanujan's mock
theta functions.

\subsection{Renormalization}\label{sec:Renormalization}\label{sec:renormalization}
This section defines the $q$-series construction called
\emph{renormalization}.

\begin{definition}
A \emph{$q$-hypergeometric series} is a sum of the form
$H(q)=\sum_{n=0}^\infty H_n(q)$ where $H_n(q) \in \Q(q)$ and
$H_{n+1}(q)/H_n(q)= R(q,q^n)$ for all $n\ge0$ for some fixed rational
function $R(q, r) \in \Q(q, r)$.
\end{definition}

Renormalization is a construction
used for extending $H(q)$ to the region $\abs{q}>1$.

\begin{definition}(Renormalization) Let 
$\varphi(z) = \sum_{n\in \Q} h_{n} K_0(2\pi \lambda \abs{n} y) e^{2\pi
    i \lambda n x}$ for some $\lambda \in \Q^+$.  Consider a
  $q$-hypergeometric series $H(q)=\sum_{n=0}^\infty H_n(q)$ that
  converges for $\abs{q}<1$ such that $H(q) = \sum_{n\in \Q^\pm } h_n
  q^{ \lambda \abs{n} + \alpha }$ for some $\alpha \in \Q$.  When
  $\lim_{n\to \infty} H_n(q^{-1})$ converges as a power series in $q$
  with $\abs{q}<1$, define the series $\cS[H](q)$ and $\cG[H](q)$ by
\begin{equation}\label{eqn:sumsoftails}
\cS[H](q) := \sum_{n=0}^\infty \( H_n(q^{-1}) - H_\infty(q^{-1})\) - \cG[H](q)
\end{equation}
where $H_{\infty}(q^{-1}) := \lim_{n\to \infty} H_{n}(q^{-1})$ such that 
\begin{enumerate}
\item  $\cG[H](q)$ vanishes to infinite order at every root of unity $q$ where $H(q^{-1})$ converges. 
\item $\cS[H](q)$ converges for $\abs{q}<1$ and  $ \cS[H](q) = -\sum_{n \in \Q^{\mp} } h_n q^{\lambda \abs{n}- \alpha}$. 
\end{enumerate}
The term $\cS[H](q)$ is called the \emph{shadow} of $H$ and
$\cG[H](q)$ is called the \emph{ghost} of $H$.
\end{definition}

\begin{remark}
In certain cases, the construction in \eqref{eqn:sumsoftails} is known
as the ``sums of tails''.  It appears in Ramanujan's work, as well as
in a number of other papers. See \cite{andrewsEuler, AGL, AJO, BK2,
  zagierStrange}, for example.
\end{remark}

\begin{example}
Andrews, Dyson, and Hickerson introduced the series 
\begin{align}
\sigma^*(q):= &  \sum_{n=0}^\infty S^*(n) q^n = 2 \sum_{n=1}^\infty \frac{(-1)^n q^{n^2}}{(q;q^2)_n}  \label{eqn:sigma*} 
=& -2q -2q^{2} + \cdots + 2q^{66} - 2q^{67} - 4q^{70} + \cdots  
\end{align}
and showed that $S^*(n) = T(-24n+1)$. 

The following calculations 
demonstrate that $\cS[\sigma](q)$ is equal to  $-\sigma^*(q)$.
In the notation of \eqref{eqn:sumsoftails},  
$\sigma(q) = \sum_{n=0}^\infty \sigma_n(q)$ with 
$\sigma_n(q) = \frac{q^{\frac{1}{2}n(n+1)}}{(-q;q)_n}$.
So  $\sigma_n(q^{-1})  = \frac{1}{(-q;q)_n}$ and $\sigma_\infty(q^{-1})  = \frac{1}{(-q;q)_\infty}$.
Then
\begin{align} \label{eqn:rama1}
\sum_{n=0}^\infty \( \sigma_n(q^{-1}) - \sigma_\infty(q^{-1})\) &= \sum_{n=0}^\infty \(\frac{1}{(-q;q)_n} - \frac{1}{(-q;q)_\infty}\)
 \\  &= 2\sum_{n=1}^\infty \frac{(-1)^{n-1} q^{n^2}}{(q;q^2)_n}
-\frac{1}{(-q;q)_\infty} \(\sum_{n=1}^\infty \frac{q^n}{1-q^n}\) \nonumber
\end{align}
(for the second equality see, for instance, (3.27) of \cite{AGL}). 
The first term on the right hand side is $-\sigma^*(q)$.
The second  term is $\cG[\sigma](q)$, which vanishes to infinite order at all roots of unity where $\sigma(q^{-1})$ exists. 
Thus $\cS[\sigma](q) = - \sigma^*(q)$. 
\end{example}

\subsection{Quantum Modular Forms}\label{sec:QMF}

Andrews, Dyson, Hickerson, and Cohen \cite{andrewsEuler, ADH,  cohen} 
established the following
identities for $\sigma(q)$ and $\sigma^*(q)$  and $\abs{q}<1$
which allows evaluation for $q$ an arbitrary root of unity
\begin{align}
\sigma(q) = & 1+ \sum_{n=0}^\infty q^{n+1} (q-1) (q^2-1) \cdots (q^n-1) \label{eqn:sigma_rootsOfUnity}\\
\sigma^*(q) =& -2 \sum_{n=0}^\infty q^{n+1} (1-q^2) (1-q^4) \cdots (1-q^{2n}) \label{eqn:sigma*_rootsOfUnity}
\end{align}
Cohen observed that for every root of unity $q$, 
$$
\sigma(q^{-1}) = - \sigma^*(q).
$$
Define $$Q_\sigma(z) := \begin{cases} q^{1/24} \sigma(q) & z \in \H \cup \Q \\ 
- q^{1/24} \sigma^*(q^{-1}) & z \in \H^- \cup \Q \end{cases}$$
with $q := e^{2\pi i z}$. 
The above equality shows how $\sigma$ can ``leak'' from  $\bH$ to $\bH^-$ 
through the dense set of ``holes'' $\bQ$ on the real axis.

Using the results of \cite{ADH, cohen} and joint results with Lewis
\cite{LewisZagier}, Zagier \cite{zagierQuantum} showed that, for
$x\in\bQ$,
$$Q_\sigma(x +1) -e^{\frac{2\pi i}{24}} Q_\sigma(x) = 0 \ \ \ \text{ and } \ \ \ 
\frac{1}{2x+1} Q_\sigma \( \frac{x}{2x+1}\) - e^{\frac{2\pi i}{24}} Q_\sigma(x) = h(x)$$
where $h:\R \to \C$ is $C^\infty$ on $\R$ and real-analytic except at $x = -1/2$.

Motivated by the pair $(\sigma, \sigma^*)$, a few other examples, 
 Zagier \cite{zagierQuantum} gave a notation of ``quantum modular form''.
However, there was no precise definition, so that more examples can be considered.
Here, we give a definition of quantum modular form that captures some key properties mentioned 
in \cite{zagierQuantum} and is precise for our purposes.
\begin{definition}
\label{def:quantum}
Let $k\in \frac{1}{2} \Z$ , $S \subset \bQ$ a dense subset, $\Gamma
\subset SL_2(\bZ)$ a subgroup of finite index and $\chi: \Gamma
\longrightarrow \bC^\times$ a finite order character.  A function $f:
S \longrightarrow \bC$ is a \emph{quantum modular form} of weight
$k$ and character $\chi$ with respect to $\Gamma$ if for every $\gamma
\in \Gamma$, the period function $h_\gamma:\R \to \C$ given by
\begin{equation}
\label{eq:period_function}
h_\gamma(x): = \chi(\gamma) f(x)  - (f \mid_{k, \gamma}) (x) 
\end{equation}
is $C^\infty$ or real-analytic at all but a finite set of points on $\bR$.
Here $\mid_{k, \gamma}$ is the weight $k$ slash operator 
$$
(f \mid_{k, \gamma}) (x) = f(\gamma x)(cx+d)^{-k}
$$
when $\gamma=\begin{bmatrix} a & b \\ c & d\end{bmatrix}$.
\end{definition}

\section{Statement of Results}\label{sec:Results}

In Section \ref{sec:Intro} the pair $(\sigma, \sigma^*)$ is used to
illustrate renormalization of $q$-hypergeometric series.  This section
contains a second example associated to a Maass waveform of eigenvalue
$1/4$.  Examples associated with mock theta functions and holomorphic
modular forms are given in Part II of this series \cite{lnrPart2}.

Let $K = \Q(\sqrt{2})$.  For ideals $\fa \subset \Z[\sqrt{2}]$ define 
\begin{equation}
\chi_{W}(\fa) := \begin{cases} 1 & \cN\fa \equiv \pm 1 
\pmod{16} \\ -1 & \cN\fa \equiv \pm 7 \pmod{16} \\
0 & \text{otherwise} \end{cases}, 
\end{equation}
where $\cN\fa=\cN_{\bQ(\sqrt{2})/\bQ}(\fa)$ denotes the usual ideal norm.
This is a Hecke character and hence it can be associated to a Maass waveform:

\begin{theorem}\label{thm:MaassW} 
For $z=x+iy$ in the Poincar\'{e} upper half plane $\bH$, define 
\begin{equation} 
\label{eqn:maass3b} \varphi_{0,W}(z) =  y^{1/2} \sum_{n\in\bZ} T_W(n) K_0(2\pi |n|y/8) e^{2\pi i nx/8},
\end{equation}
where $K_0(y)$ is the $K$-Bessel function.
Then $\varphi_{0,W}(z)$ is a Maass waveform. More precisely: 
\begin{enumerate}
\item $\varphi_{0,W}$ is an eigenfunction of the hyperbolic Laplacian $\Delta=-y^2 \left(\frac{\partial^2}{\partial x^2}+ \frac{\partial^2}{\partial y^2} \right)$ with eigenvalue $\lambda=1/4$;
\item $\varphi_{0,W}\left( \frac{-1}{4z} \right)=\overline{\varphi_{0,W}(z)}$;
\item for each $k \in \bZ$, there exists a multiplier system $\nu_k$ for $\Gamma_0(4)$, uniquely determined by
\[
\nu_k\left( \begin{bmatrix} 1 & 1 \\ 0 & 1 \end{bmatrix} \right) = e^{2\pi i/8}, \ \  \ \ 
\nu_k\left( \begin{bmatrix} 3 & -1 \\ 4 & -1 \end{bmatrix} \right) = 1, \ \ \ \
\nu_k\left( \begin{bmatrix} -1 & 0 \\ 0 & -1 \end{bmatrix} \right) = (-1)^k, \ \ \ \
\]
such that $\varphi_{0,W}(\gamma z)=\nu_0(\gamma)\varphi_{0,W}(z)$ for every $z\in\bH$ and every $\gamma\in\Gamma_0(4)$;
\item $\varphi_{0, W}(z)$ vanishes at the cusps $0$ and $\infty$, but does \textbf{not} vanish at the cusp $1/2$. 
\end{enumerate}
\end{theorem}

Define the following $q$-hypergeometric series 
\begin{align}
\label{eq:Wq_definition}
W(q):= \sum_{n=1}^\infty W_{n}(q) :=&  \sum_{n=1}^\infty \frac{(-1;q^2)_n (-1)^n q^n}{(q;q^2)_n} \\
=& -2 q - 2q^3 + 2q^4 + 2q^6 + 2q^8 - 2q^{9} +  \cdots - 2q^{97}  - 4q^{99} + 4q^{100} + \cdots  \nonumber 
\end{align}

The following theorem which
relates the pair of $q$-hypergeometric series $(W(q),\cS[W](q))$ to
the Maass waveform $\varphi_{0,W}$ and thus closes the analogy with
the pair $(\sigma, \sigma^*)$. Moreover, the theorem shows that the 
renormalization construction is an involution.

 \begin{theorem}\label{thm:W2}\label{thm:realquadratic_coefficients}
In the notation of Theorem \ref{thm:MaassW}, we have
$$
W(q)
 = \sum_{n>0} T_W(8n-1) q^n
$$
and 
$$\cS[W](q) = - \sum_{n<0} T_W(8n + 1) q^n$$
where 
\begin{equation}\label{eqn:W2star}
\cS[W](q) := \sum_{n=1}^\infty \( 
W_{n}(q^{-1}) - W_{\infty}(q^{-1})
\)
- \frac{(-1;q^2)_\infty}{(q;q^2)_\infty} \( \frac{1}{2} - \sum_{n=1}^\infty \frac{q^n}{1+q^n}\),
\end{equation}
we have 
\begin{align*}
\cS[W](q):=& -\sum_{n=0}^\infty \frac{(q;q^2)_n (-1)^nq^n}{(-q^2;q^2)_n}  \\
=&-( 1 -q + 2q^2 -q^3 -2q^5 + 3q^6 -2q^9 + q^{10} + \cdots - q^{91} - 2q^{95} + 2q^{96} + 2q^{100} + \cdots )
\end{align*}
and 
$$
\cG[W](q) = - \frac{(-1;q^2)_\infty}{(q;q^2)_\infty} \( \frac{1}{2} - \sum_{n=1}^\infty \frac{q^n}{1+q^n}\).$$

Moreover, 
the \emph{renormalized shadow} of the series $\cS[W](q)$ is $W(q)$, i.e.
$\cS[\cS[W]](q) = W(q).$
More precisely, with $\cS[W]_n(q) := \frac{(-1)^{n+1} q^n (q;q^2)_n}{(-q^2;q^2)_n}$
so that $\cS[W](q) = \sum_{n=0}^\infty \cS[W]_{n}(q)$, 
\begin{align*}
\sum_{n=0}^\infty \( \cS[W]_n(q^{-1}) - \cS[W]_\infty(q^{-1})\) =& 
-\sum_{n=0}^\infty \( \frac{(q;q^2)_n}{(-q^2;q^2)_n} - \frac{(q;q^2)_\infty}{(-q^2;q^2)_\infty}\) \\
=& W(q) + \cG[\cS[W]](q)
\end{align*}
where 
$$\cG[\cS[W]](q) = \frac{(q;q^2)_\infty}{(-q^2;q^2)_\infty} \( \sum_{n=1}^\infty \frac{q^{2n}}{1-q^{2n}} + \sum_{n=1}^\infty \frac{q^{2n+1}}{1+q^{2n+1}}\).$$
\end{theorem}

\begin{remark}
For completeness, we show here that $\cS$ is an involution for the pair $(\sigma, \sigma^*)$.  
In the notation of \eqref{eqn:sumsoftails}, 
$\cS[\sigma](q) =- \sum_{n=0}^\infty \sigma_{n}^*(q)$, with $\sigma_n^*(q) = \frac{2(-1)^{n}q^{n^2}}{(q;q^2)_n}$.  
So that $\sigma_n^*(q) = \frac{-2}{(q;q^2)_n}$. 
From \cite{AJO} (see also \cite{andrewsEuler} and \cite{cj} for proofs)
\begin{align}
\sum_{n=0}^\infty \( \sigma_n^*(q^{-1}) - \sigma_\infty^*(q^{-1}) \) =&
-2\sum_{n=1}^\infty \( \frac{1}{(q;q^2)_n} - \frac{1}{(q;q^2)_\infty}\)  \nonumber  \\=& 
\sum_{n=0}^\infty \frac{q^{\frac{1}{2}n(n+1)}}{(-q;q)_n}+ \frac{2}{(q;q^2)_\infty} \( -\frac{1}{2} + \sum_{n=1}^\infty \frac{q^n}{1-q^n}\). \label{eqn:dyson_example}
\end{align}
So, it is clear that $\cS[\cS[\sigma]] (q) = \sigma(q).$
That is the shadow of $\sigma(q)$ is $-\sigma^*(q)$ and vice-versa. 
\end{remark}

The following theorem gives the quantum modular form associated with the pair $(W, \cS[W])$. 
\begin{theorem}\label{thm:realquadratic_quantum}
For $\iota \in \{ 0, 1/2, \infty \}$, define the following dense subsets of $\bQ$
$$
S_\iota = \{ x \in \bQ: \gamma \cdot x = \iota \text{ for some } \gamma \in \Gamma_0(4) \}.
$$
Then the following function $f_W: S_\infty \cup S_0 \longrightarrow \bC$ is a quantum modular form with respect to $\Gamma_0(4)$ of weight 1 and character $\nu_1$ as in Theorem \ref{thm:MaassW}:
\begin{equation}
f_W(x) 
:= 
\left\{ 
\begin{array}{cc}
q^{-\tfrac{1}{8}} W(q) & \text{if } x \in S_\infty \\
q^{-\tfrac{1}{8}} \cS[W](q^{-1}) & \text{if }x \in S_0. 
\end{array}
\right.
\label{eqn:qm_def}
\end{equation}
\end{theorem}

\begin{remark}
The set $S_\infty$ consists of fractions $\tfrac{a}{4b}$ with $a, b \in \bZ$ and $\gcd(a, 4b) = 1$. 
Then it is not hard to see that the $q$-hypergeometric expression of $W(q)$ in \eqref{eq:Wq_definition} makes sense when $q = e^{2\pi i x}$ with $x \in S_\infty$. 
Similarly, $\cS[W](q^{-1})$ can be evaluated at $q = e^{2\pi i x}$ when $x \in S_0$. 
Thus, $f_W(x)$ is defined on $S = S_\infty \cup S_0 \subset \bQ$. 
\end{remark}

It is unlikely that there exist expressions for $W(q)$ and
$\cS[W](q)$, similar to those for $\sigma(q)$ and $\sigma^*(q)$ in
\eqref{eqn:sigma_rootsOfUnity} and \eqref{eqn:sigma*_rootsOfUnity},
which exist at every root of unity.  The reason is that the Maass
waveform $\varphi_{0, W}(z)$ is not a cusp form and it is not clear
how to define $f_W(x)$ when $x \in \bQ - S = S_{1/2}$.

In Section \ref{sec:background} we collect some results on
$q$-hypergeometric series and Hecke $L$-functions which are important
for our study.  Section \ref{sec:realquadratic} contains the proof of
Theorems \ref{thm:MaassW}, \ref{thm:W2},and  \ref{thm:realquadratic_quantum}.

In Section \ref{sec:Other} we discuss ``renormalization'' applied to a
number of other examples which exhibit interesting structure, but do
not seem to fit into any known theory of non-holomorphic modular
forms. It would be interesting if these examples were explored
further.

\section{Preliminaries}\label{sec:background}

\subsection{$q$-hypergeometric series}\label{sec:q-seriesBack}

This subsection collects some results dealing with $q$-hypergeometric series. 
Throughout we adopt Fine's \cite{fine} notation for the basic hypergeometric series 
$$F(a, b;t) = F(a, b; t:q):= \sum_{n=0}^\infty \frac{(aq;q)_n}{(bq;q)_n} t^n.$$
The following identity is (6.3) of \cite{fine}
\begin{equation}\label{eqn:fine6.3}
F(a, b;t) = \frac{1-b}{1-t} F\( \frac{at}{b}, t;b\).
\end{equation}

Finally, the following identity is due to  Ramanujan (Entry 1.7.2 of Ramanujan's Lost notebook book  volume II \cite{ABII}):
\begin{equation}\label{eqn:entry1.7.2}
\sum_{n=0}^\infty  \frac{(-aq/b)_n b^n (-1)^n q^{\frac{1}{2}n(n+1)}}{(-b;q)_{n+1}} = 
\sum_{n=0}^\infty \frac{(-b)^n (-q;q)_n (-aq/b;q)_n}{(aq;q^2)_{n+1}}.
\end{equation}

The next two theorems will be used to carry out the renormalization in the next three sections of this paper. 
The theorems are due to Andrews, Jim\'enez-Urroz, and Ono \cite{AJO}.
\begin{theorem}[Theorem 1 of \cite{AJO}]\label{thm:AJO}
\begin{align*}
\sum_{n=0}^\infty & \( \frac{(t;q)_\infty}{(a;q)_\infty} - \frac{(t;q)_n}{(a;q)_n}\) 
= \sum_{n=1}^\infty \frac{(q/a;q)_n}{(q/t;q)_n} \(\frac{a}{t}\)^n \\
& \hspace{.7in} + \frac{(t;q)_\infty}{(a;q)_\infty} \( \sum_{n=1}^\infty \frac{q^n}{1-q^n} + \sum_{n=1}^\infty 
\frac{q^nt^{-1}}{1-q^nt^{-1}} - \sum_{n=0}^\infty \frac{tq^n}{1-tq^n} - \sum_{n=0}^\infty \frac{at^{-1}q^n}{1-at^{-1}q^n}\).
\end{align*}
\end{theorem}


\subsection{Hecke Characters and Hecke $L$-functions}\label{sec:heckeBack}


This subsection recalls some standard facts about Hecke characters of number fields and their associated Hecke $L$-functions. Throughout we adopt Neukirch's \cite{Neukirch} notation. For our needs in the paper, we restrict our discussion to real quadratic number fields.

Let $K$ be a real quadratic field and $\cO$ its ring of integers. Then $K\otimes_\bQ\bR\cong\bR^2$ via the isomorphism $\alpha\otimes1\leftrightarrow(\alpha_1:=\sigma_1(\alpha),\alpha_2:=\sigma_2(\alpha))$ for every $\alpha\in K$, where $\sigma_1$ and $\sigma_2$ are the real embeddings of $K$. Put $|\alpha|:=(|\alpha_1|,|\alpha_2|)$ for $\alpha\in K$ and $N(\beta_1,\beta_2):=\beta_1\beta_2$ for $\beta_1,\beta_2\in\bR$.

Let $\fm$ be an integral ideal of $K$. We write $J^\fm$ for the group of fractional ideals relatively prime to $\fm$. A homomorphism 
$$
\chi: J^\fm\to S^1 :=\{z\in\bC:|z|=1\}
$$ 
is called a \emph{character mod $\fm$}.

\begin{definition}
\begin{enumerate}
\item[${\rm (i)}$] A \emph{Hecke character mod $\fm$} is a character $\chi$ mod $\fm$ for which there exists a pair of characters
$$
\chi_f:(\cO/\fm)^*\to S^1, \quad \chi_\infty:(K\otimes\bR)^*\to S^1
$$
such that $\chi_\infty$ is continuous (with respect to the real topologies on both sides) and 
$$
\chi((\alpha))=\chi_f(\alpha)\chi_\infty^{-1}(\alpha\otimes1)
$$
for every $\alpha\in\cO$ relatively prime to $\fm$. We call $\chi_f$ (resp. $\chi_\infty$) the \emph{finite (resp. infinite) component} of $\chi$.
\item[${\rm (ii)}$] With the same notation as in ${\rm (i)}$, the continuity of $\chi_\infty$ imply that 
$$
\chi_\infty^{-1}(\alpha\otimes1)=N(\alpha^p|\alpha|^{p+iq})
$$
for every $\alpha\in K^*$, where $p=(\eps_1,\eps_2)$ with $\eps_i\in\{0,1\}$ and $q\in\bR^2$. We say that \emph{$\chi$ has infinity type $(++)$ (resp. $(+-)$, resp. $(--)$)} if $\eps_1=\eps_2=0$ (resp. $\{\eps_1,\eps_2\}=\{0,1\}$, resp. $\eps_1=\eps_2=1$).
\item[${\rm (iii)}$] A Hecke character $\chi$ mod $\fm$ as in ${\rm (i)}$ is \emph{primitive} if $\chi_f$ does not factor through $(\cO/\fm')^*$ for any proper divisor $\fm'$ of $\fm$. In this case, $\fm$ is called the \emph{conductor} of $\chi$.
\end{enumerate}
\end{definition}

Now fix an integral ideal $\ff$ and a primitive Hecke character $\chi$ of conductor $\ff$. Suppose 
$$
\chi_\infty^{-1}(\alpha\otimes1)=N(\alpha^p|\alpha|^{p+iq})
$$
for every $\alpha\in K^*$. Put 
\begin{align*}
L_X(s_1,s_2) &= \pi^{-s_1/2}\Gamma(s_1/2)\cdot \pi^{-s_2/2}\Gamma(s_2/2) & (s_1,s_2\in\bC)  \\
L_\infty(\chi,s) &= L_X((s,s)+p-iq) & (s\in\bC) .
\end{align*}

\begin{definition}
The \emph{completed Hecke $L$-function associated to (primitive) $\chi$} is the function
$$
\Lambda(\chi,s):=\( |d_K| \cdot N\ff  \)^{s/2}L_\infty(\chi,s)\cdot
L(\chi, s)
$$
where 
$$
L(\chi, s) =  \sum_{\fa\subset \cO} \frac{\chi(\fa)}{(\cN\fa)^{s}},
$$
$\cN\fa$ denotes the ideal norm of $\fa$ and $d_K$ is the discriminant of $K$.
\end{definition}

\begin{theorem}
Suppose $\chi$ is a primitive Hecke character mod $\ff$. Then $\Lambda(\chi,s)$ converges absolutely (uniformly on compacta) for $\Re(s)>1$. It has a meromorphic continuation to $s\in\bC$ with a functional equation
$$
\Lambda(\chi,s)=W(\chi)\cdot \Lambda(\overline{\chi},1-s)
$$
where $W(\chi)$ is the root number associated to $\chi$ which (can be computed explicitly using Gauss sum and) satisfies $|W(\chi)|=1$. If $\ff\neq1$ or $p\neq0$, then $\Lambda(\chi,s)$ is an entire function of $s$.
\end{theorem}

\begin{remark}
For the explicit computation of the root number, see for example \cite[Chapter ${\rm VII}$, Section $8$]{Neukirch}.
\end{remark}

Now suppose $\chi_f$ is trivial on $\cO^*$, $q = (0, 0)$ and $\chi$ has infinity type $(++)$. Then $\chi$ has finite order and is a character of the ray class group of $K$ of modulus $\ff$.
By class field theory, it is a character of $\Gal(K_\ff/K)$, where $K_\ff$ is the ray class field of $K$ of modulus $\ff$. 
Let $\overline{\bQ}$ be a fixed algebraic closure $\bQ$ and $M \subset \overline{\bQ}$ be the normal closure of $K_\ff$ over $\bQ$.
Denote the induced representation of $\chi: \Gal(K_\ff/ K) \longrightarrow \bC^\times$ to $\Gal(\overline{\bQ}/\bQ)$ by
$$
\rho_\chi: \Gal(\overline{\bQ}/\bQ) \longrightarrow \mathrm{GL}_2(\bC),
$$ 
whose kernel contains $\Gal(\overline{\bQ}/M)$. 
Let $L(\rho_\chi, s)$ be the Artin L-function associated to $\rho_\chi$. 
Then we have
$$
L(\chi, s) = L(\rho_\chi, s).
$$
Since the infinity type of $\chi$ is $(+ +)$, we deduce that $\rho_\chi$ is an even representation, i.e. $\det(\rho_\chi(c)) = 1$ with $c \in \Gal(\overline{\bQ}/\bQ)$ complex conjugation. Thus, it can be associated a Maass waveform on $\Gamma_0(N)$ of eigenvalue $1/4$ by Langlands' theory of solvable base change \cite{Langlands}, where $N$ is the conductor $\rho_\chi$. 
If $\rho_\chi$ is irreducible (resp.\ reducible), then the associated Maass waveform vanishes (resp.\ does not vanish) at all the cusps of $\Gamma_0(N)$. 
In that case, we call it a Maass cusp form (resp.\ Maass-Eisenstein series).

\subsection{Period functions of Maass wave forms}\label{sec:periodFunction}

This section contains results on period functions of Maass wave forms.  The calculations are used to establish the quantum modularity of the function $f_W$ in Theorem \ref{thm:realquadratic_quantum}.
The results are essentially contained in \cite{LewisZagier} (see also \cite{BLZ1, BLZ2} for much more about period functions of Maass wave forms). However, we were unable to locate the precise proposition we needed.
So we have included the details.

Throughout this section, let 
$u(z)$ be a Maass form of spectral parameter $s$, $s = 1/2$, i.e.\ it is an eigenfunction of the Laplacian with eigenvalue $s(1-s) = \frac{1}{4}$.
Write its Fourier expansion in the form
$$
\uu (z) = \sqrt{y} \sum_{n \neq 0} \cn K_{s - 1/2}(2 \pi |n| y ) e^{2 \pi i n x}.
$$
Here $z = x + iy$ and $K_\nu(t)$ is the modified Bessel function satisfying the differential equation
$$
t^2 \frac{d^2 f}{dt^2} + t \frac{df}{dt} - (t^2 + \nu^2) f = 0.
$$
Notice that $K_\nu(t) = K_{- \nu}(t)$. At $\nu = 1/2$, the modified Bessel function can be written as
$$
K_{1/2}(t) = \sqrt{\frac{\pi}{2}} t^{-1/2} e^{-t}.
$$
Furthermore, the first derivative of $K_\nu(t)$ is $-(K_{\nu - 1}(t) + K_{\nu + 1}(t))/2$. In particular, $K'_0(t) = -K_1(t)$. 
Define the real-analytic function $R_\xi(z)$ to be
$$
R_\xi(z) := \frac{y}{(x - \xi)^2 + y^2} = \frac{i}{2} \left( \frac{1}{z - \xi} - \frac{1}{\overline{z} - \xi} \right).
$$
Then the function $(R_\xi)^s$ is an eigenfunction of the Laplacian of the same eigenvalue, $s(1-s)$. 
For two real-analytic functions $f(z), g(z)$ on the upper half plane, define their complex Green's form to be
$$
[f, g] := \frac{\partial f}{\partial z} g dz + f \frac{\partial g}{\partial \overline{z}} d\overline{z}.
$$
\begin{proposition}\label{prop:termbyterm}
In the notation above, for $z\in \H$ 
$$\sum_{n>0} A_n e^{2\pi i n z} = -\frac{2}{\pi}  \int_z^{i\infty} [u(\tau), R_z(\tau)^{1/2}]$$
and for 
$z \in \H^-$ 
$$\sum_{n<0} A_n e^{2\pi i n z} = - \frac{2}{\pi} \int_{\overline{z}}^{i\infty} [R_z(\tau)^{1/2}, u(\tau)].$$
\end{proposition}

\begin{proof} 
Define 
\begin{align*}
f(\xi) &:= \sum_{n \ge 1} \cn e^{2 \pi i n \xi}, \\
f_1(\xi) &:= \int^{i \infty}_{\xi} [ u, R^s_\xi] (z).
\end{align*}
We want to show that the functions $f$ equals $f_1$ up to a constant.

First, it is not hard to see that $f_1(\xi)$ is holomorphic in $\xi$ (see pg 212 of \cite{LewisZagier}).
So it is enough to consider $\xi = i y_0$ with $y_0 \in \mathbb{R}^+$. 
In this case, we can write 
\begin{align*}
f_1(\xi) &= 
\int^{i\infty}_{i y_0} \frac{1}{2} \left( \frac{\partial \uu}{\partial x} - i \frac{\partial \uu}{\partial y} \right) \left( \frac{y}{y^2 - y_0^2}\right)^{1/2} dz 
+ 
\uu \frac{\partial }{\partial \overline{z}} (R^{1/2}_\xi) d \overline{z} \\
&= 
 \int^{i\infty}_{i y_0} \frac{1}{2} \left( \frac{\partial \uu}{\partial x} - i \frac{\partial \uu}{\partial y} \right) \left( \frac{y}{y^2 - y_0^2}\right)^{1/2} dz 
+
\uu  \frac{-i}{4\sqrt{y}} \frac{(y- y_0)^2}{(y^2 - y^2_0)^{3/2}} d \overline{z} \\
&= 
i \int^{\infty}_{y_0} \frac{1}{2} \left( \frac{\partial \uu}{\partial x} - i \frac{\partial \uu}{\partial y} \right) \left( \frac{y}{y^2 - y_0^2}\right)^{1/2} dy
+
\uu  \frac{i}{4\sqrt{y}} \frac{(y- y_0)^2}{(y^2 - y^2_0)^{3/2}} dy \\
&= i(
g_1(y_0) + g_2(y_0) + g_3(y_0)),
\end{align*}
where 
\begin{align*}
g_1(y_0) &:= \frac{1}{2} \int^\infty_{y_0} \frac{\partial \uu}{\partial x} \mid_{x = 0} \cdot \left( \frac{y}{y^2 - y^2_0} \right)^{1/2}, \\
g_2(y_0) &:= \frac{-i}{2} \int^\infty_{y_0} \frac{\partial \uu}{\partial y} \mid_{x = 0} \cdot \left( \frac{y}{y^2 - y^2_0}\right)^{1/2}, \\
g_3(y_0) &:= \frac{i}{4}  \int^\infty_{y_0} \uu \frac{(y - y_0)^2}{\sqrt{y} (y^2 - y^2_0)^{3/2}} dy.
\end{align*}
After substituting in the Fourier expansion of $\uu(z)$ into the integrals defining $g_i(y_0)$, we could interchange the integral and summation and evaluate them.
Equation  (13) on page 129 of \cite{bateman} states:
$\mathrm{Re}(\mu) > -1$ and $\mathrm{Re}(h) > 0$, then
\begin{equation}
\int^\infty_a \frac{y^{1 - \nu}}{ (y^2 - a^2)^\mu} K_\nu(hy) dy
= 
2^{-\mu} a^{1 -\mu - \nu} h^{\mu - 1} \Gamma(1-\mu) K_{1-\mu - \nu}(ha).
\label{Ktransform}
\end{equation}
Using this, we could evaluate $g_j(y_0)$ as follows.
\begin{align*}
g_1(y_0) &= \pi i \sum_{n \neq 0}  n \cn \int^\infty_{y_0} K_0(2 \pi |n| y)  
 \frac{y}{(y^2 - y^2_0)^{1/2}} dy \\
&= \frac{\pi i \sqrt{y_0}}{2} \sum_{n \neq 0} \cn \frac{n}{\sqrt{|n|}} K_{1/2}(2 \pi |n| y_0) \\
&= \frac{\pi i}{4} \sum_{n \neq 0} \cn \frac{n}{|n|} e^{-2 \pi |n| y_0}.
\end{align*}
For $g_2(y_0)$ and $g_3(y_0)$, we can write
\begin{align*}
g_2(y_0) &= \frac{-i}{2} \sum_{n \neq 0} \cn \int^\infty_{y_0} \frac{K_0(2 \pi |n| y) }{ 2(y^2 - y^2_0)^{1/2}} + \frac{(-2 \pi |n|) K_1(2 \pi |n| y) y}{(y^2 - y^2_0)^{1/2}} dy, \\
g_3(y_0) &= \frac{i}{4} \sum_{n \neq 0} \cn \int^\infty_{y_0} \frac{K_0(2 \pi |n| y) }{(y^2 - y^2_0)^{1/2}} - \frac{2 y_0( y - y_0) K_0(2 \pi |n| y)}{(y^2 - y^2_0)^{3/2}} dy.
\end{align*}
Adding them together, the first term cancels and the sum can be re-written as
\begin{align*}
g_2(y_0) + g_3(y_0) 
= &
\frac{-i}{2} \sum_{n \neq 0} \cn \int^\infty_{y_0} 
(-2 \pi |n|) K_1(2 \pi |n| y) \left( \frac{ y - y_0 }{ y + y_0} \right)^{1/2}
+ 
K_0(2 \pi |n| y) \frac{y_0(y - y_0)}{(y^2 - y^2_0)^{3/2}} dy \\
& 
+ \pi i y_0 \sum_{n \neq 0} |n| \cn \int^\infty_{y_0} \frac{K_1(2 \pi |n| y) }{(y^2 - y^2_0)^{1/2}} dy.
\end{align*}
The first term is exact and integrates to $K_0(2 \pi |n| y)\left( \frac{ y - y_0 }{ y + y_0} \right)^{1/2}$, which evaluates to 0 at both ends.
The second term can be evaluated using  \eqref{Ktransform}, which gives us 
$$
g_2(y_0) + g_3(y_0) = \frac{\pi i}{4} \sum_{n \neq 0} \cn e^{-2 \pi |n| y_0}.
$$
Adding this to $g_1(y_0)$, we have
$$
f_1(i y_0) = -\frac{\pi }{2} \sum_{n \ge 1} \cn e^{-2 \pi |n| y_0} =  -\frac{ \pi }{2} f(i y_0).
$$
A similar calculation yields the second claimed identity. 
\end{proof}

\section{Renormalization and the Maass waveform $\varphi_{0,W}$}\label{sec:realquadratic} 



This section contains the proofs of  the theorems stated in Section \ref{sec:Results} 
concerning the $q$-hypergeometric series $W(q)$. 

We start with Theorem \ref{thm:MaassW}. Our proof closely follows the technique that Cohen employs in \cite{cohen} to prove his Theorems $2.1$ and $3.1$ therein. First, an easy computation using the formalism of Hecke characters and their associated Hecke $L$-functions allows us to establish the following 

\begin{proposition}
\begin{enumerate}
\item $\chi_W$ (resp. $\chi_W\cdot \left( \frac{-1}{\cN} \right)$) is a primitive Hecke character of conductor $\ff=(\sqrt{2})^5$ and infinity type $(++)$ (resp. $(--)$) on $\bQ(\sqrt{2})$.
\item Let $\Lambda(s)$ (resp. $\Lambda_1(s)$) be the completed Hecke $L$-functions associated to $\chi_W$ (resp. $\chi_W\cdot \left( \frac{-1}{\cN} \right)$), i.e.
\begin{align*}
\Lambda(s) &= 16^s\cdot\pi^{-s}\Gamma\left(\frac{s}{2}\right)^2\cdot \sum_{\fa \subset \Z[\sqrt{2}]} \frac{\chi(\fa)}{(\cN\fa)^s} , \\
\Lambda_1(s) &= 16^s\cdot\pi^{-s-1}\Gamma\left(\frac{s+1}{2}\right)^2\cdot \sum_{\fa \subset \Z[\sqrt{2}]} \frac{\chi_W(\fa)\cdot \left( \frac{-1}{\cN\fa} \right)}{(\cN\fa)^s}.
\end{align*}
Then $\Lambda(s)$ (resp. $\Lambda_1(s)$) converges absolutely (uniformly on compacta) for $\Re(s)>1$, extends to an entire function on $\bC$ which is bounded in vertical strips, and has functional equation $\Lambda(s)=\Lambda(1-s)$ (resp. $\Lambda_1(s)=\Lambda_1(1-s)$).
\end{enumerate}
\end{proposition}

\begin{proof}[Proof of Theorem \ref{thm:MaassW}]

Statement $(1)$ is immediate from the standard properties of the $K$-Bessel function $K_0$.

For statement $(2)$, we first observe that $\varphi_0:=\varphi_{0,W}$ and $\varphi_1(z)=\varphi_{1,W}(z):= \overline{\varphi_{0,W}\left(\frac{-1}{4z}\right)}$ are both eigenfunctions of the hyperbolic Laplacian $\Delta$ with the same eigenvalue. Hence, by standard PDE arguments (see e.g. \cite[Lemma 1.9.2 page 109]{Bump}), it suffices to show that 
\begin{equation}\label{eq:MaassWeq1}
\varphi_0(it)=\varphi_1(it)
\end{equation}
and 
\begin{equation}\label{eq:MaassWeq2}
\left[\frac{\partial}{\partial x} \varphi_0(x+it)\right]_{x=0} = \left[\frac{\partial}{\partial x} \varphi_1(x+it)\right]_{x=0}.
\end{equation}

To see why (\ref{eq:MaassWeq1}) holds true, observe that for $\sigma_0>1/2$ we have
\[
\varphi_0(it)=\frac{1}{8\sqrt{2}\pi i} \int_{\sigma_0-i\infty}^{\sigma_0+i\infty} 2^{-s}\Lambda(s+1/2)t^{-s}ds.
\]
Since $\Lambda(s)$ is entire and bounded in vertical strips, we can shift the line of interation to $\Re(s)=-\sigma_0$ and then apply the functional equation of $\Lambda(s)$ to derive (\ref{eq:MaassWeq1}).

The equality (\ref{eq:MaassWeq2}) follows similarly, albeit more complicated. First, we put
\[
\psi_i(t):=\left[\frac{\partial}{\partial x} \varphi_i(x+it)\right]_{x=0},
\]
and compute the Mellin transform of $\psi_0(t)$ which equals
\[
\pi i 2^{-s-1/2}\Lambda_1(s-1/2).
\]
Hence, for $\sigma_0>3/2$ we have
\[
\psi_0(t)=\int_{\sigma_0-i\infty}^{\sigma_0+i\infty} 2^{-s-3/2}\Lambda(s-1/2)t^{-s}ds.
\]
Since $\Lambda_1(s)$ is entire and bounded in vertical strips, we can shift the line of interation to $\Re(s)=2-\sigma_0$ and then apply the functional equation of $\Lambda_1(s)$ to derive (\ref{eq:MaassWeq2}).

For Statement $(3)$, we note that $\Gamma_0(4)$ is generated by $3$ elements
\[
A=\begin{bmatrix} 1 & 1 \\ 0 & 1 \end{bmatrix}, 
B=\begin{bmatrix} 3 & -1 \\ 4 & -1 \end{bmatrix},
C=\begin{bmatrix} -1 & 0 \\ 0 & -1 \end{bmatrix}.
\]
Put $D=\begin{bmatrix} 0 & -1 \\ 4 & 0 \end{bmatrix}$, so $\varphi_{0,W}(Dz)=\overline{\varphi_{0,W}(z)}$. Then $(3)$ follows from the facts that $ADAD=4B$ and that $\varphi_{0,W}(z-1)=e^{-2\pi i/8}\varphi_{0,W}(z)$, or equivalently that $\varphi_{0,W}(Az)=e^{2\pi i/8}\varphi_{0,W}(z)$.

Finally for Statement $(4)$, let $\rho_W: \Gal(\overline{\bQ}/\bQ) \longrightarrow \mathrm{GL}_2(\bC)$ be the Galois representation induced from $\chi_W$ and consider its associated Maass waveform $\tilde{\varphi}_{0, W}(z)$, which can be written as 
$$
\tilde{\varphi}_{0, W}(z) = \varphi_{0, W}(8z).
$$
Notice that the transformation identities satisfied by $\varphi_{0, W}(z)$ implies $\tilde{\varphi}_{0, W}(z)$ has level 256 and trivial character.

We have shown that $\varphi_{0, W}(z)$ transforms with respect to $\Gamma_0(4)$, which has three cusps, and vanishes at the cusps 0 and $\infty$.
To show that it does not vanish at the cusp $1/2$, it suffices to prove that $\varphi_{0, W}(8z)$
does not vanish at some cusp of $\Gamma_0(256)$, i.e. $\tilde{\varphi}_{0, W}(z)$ is a Maass-Eisenstein series.
It is not hard to see that the ray class field $K_\ff$ with $K = \bQ(\sqrt{2})$ and $\ff = (\sqrt{2})^5$ is the totally real subfield of degree 4 contained in $\bQ(\zeta_{16})$ and isomorphic to
$$
\bQ[X]/(X^4 - 4X^2 + 2).
$$
Since $K_\ff/\bQ$ is Galois and abelian, $\rho_W$ factors through $\Gal(K_\ff/\bQ)$ and is reducible. 
Thus, $\tilde{\varphi}_{0, W}(z)$ is a Maass-Eisenstein series.

Furthermore, the $L$-function $L(\rho_W, s)$ can be factored as
$$
L(\rho_W, s) = \frac{L(K_\ff, s)}{L(K, s)} = L(\chi, s) L(\overline{\chi}, s),
$$
where $L(F, s)$ is the Dedekind zeta function for a number field $F$ and $\chi: (\bZ/16 \bZ)^\times \longrightarrow \bC^\times$ is the Dirichlet character satisfying
$$
\chi(-1) = 1, \chi(3) = i.
$$
This also gives us the relationship
$$
\chi_W(n) = \sum_{d \mid n, d > 0} \chi(d) \overline{\chi} \left( \tfrac{n}{d} \right).
$$
\end{proof}

\begin{remark}
One can use the relationships among the Fourier coefficients of $\tilde{\varphi}_{0, W}(z)$ at various cusps given in \textsection 7 of \cite{GHL} to explicitly express the Fourier coefficients of $\varphi_{0, W}(z)$ at the cusp $1/2$ in terms of its Fourier coefficients at $\infty$. 
\end{remark}


Next, we consider Theorem \ref{thm:W2}.

Corson, Favero, Liesinger, and Zubairy \cite{CFLZ} defined the pair of $q$-hypergeometric series 
\begin{equation}
W_1(q):= \sum_{n=0}^\infty \frac{(q;q)_n (-1)^n q^{\frac{1}{2}n(n+1)}}{(-q;q)_n} \label{eqn:CFLZ_W1} \ \ \ \text{ and } \ \ \ 
W_2(q):= \sum_{n=1}^\infty \frac{(-1;q^2)_n (-1)^n q^n}{(q;q^2)_n}. 
\end{equation}

\begin{theorem}[Theorems 1.1 and 3.2 of \cite{CFLZ}]\label{thm:CFLZ} 
For $\abs{q}<1$ we have 
$$qW_1(q^8) + q^{-1} W_2(q^8) = \sum_{\fa \subset \Z[\sqrt{3}] } \chi_{W}(\fa) q^{\abs{N(\fa)}}.$$
 Moreover, that $qW_1(q^8) = \sum_{\begin{subarray}{c} \fa \subset \Z[\sqrt{2}] \\
N(\fa)>0 \end{subarray}} \chi_{W}(\fa) q^{N(\fa)}.$ 
\end{theorem}
\begin{remark}
As above,  our notation for the norm is slightly different. 
\end{remark}

In particular, this theorem follows easily from the following identities 
\begin{align}
W_1(q) =& \sum_{n \ge 0} \sum_{\abs{j} \le n} (-1)^{n+j} q^{2n^2+ n - j^2} (1-q^{2n+1}) \\
W_2(q) = & \sum_{n\ge 1} \sum_{-n \le j < n} (-1)^n q^{2n^2 - n - j^2 +j} (1+q^{n}).
\end{align}
These identities appear as Theorems 2.3 and 2.4 of \cite{CFLZ} and follow from the method of Bailey pairs.

We need one more result dealing with these series before we can prove Theorem \ref{thm:W2}. 
\begin{proposition}\label{prop:W1}
In the notation of Theorem \ref{thm:W2}, with $\abs{q}<1$, we have 
$$W_1(q) = \cS[W](q) 
.$$

\end{proposition}
\begin{proof}
We begin with a proof of the first claim. 
Equation \eqref{eqn:entry1.7.2} with $a = -1$ and $b = 1$ 
gives 
\begin{align*}
\frac{1}{2} \sum_{n=0}^\infty \frac{(q)_n (-1)^n q^{\frac{1}{2}n(n+1)}}{(-q)_n} 
= &
\sum_{n=0}^\infty \frac{(q^2;q^2)_n}{(-q;q^2)_{n+1}} (-1)^n \\=&
\frac{1}{1+q} F(1, -q; -1:q^2)
 \end{align*}
 Equation \eqref{eqn:fine6.3} gives 
 $$F(1, q,;-1:q^2) = \frac{1+q}{2} F(q^{-1}, -1; -q:q^2).$$
 Combining these establishes the desired identity. 
 
\end{proof}

\begin{proof}[Proof of Theorem \ref{thm:W2}]

The claim about the $q$-expansions $W(q)$ and $\cS[W](q)$ 
follow immediately from Theorem \ref{thm:CFLZ} and Proposition \ref{prop:W1}.

It is
clear that $W_{n}(q^{-1}) = (-1;q^2)_n/(q;q^2)_n$ and
$W_{\infty}(q^{-1}) = \frac{(-1;q^2)_\infty}{(q;q^2)_\infty}.$
Applying Theorem \ref{thm:AJO} with $t=-q^2$ and $a = q^3$, then
multiplying by $\frac{2}{(1-q)}$ gives \eqref{eqn:W2star}.
Proposition \ref{prop:W1} and Theorem \ref{thm:CFLZ} give the first
renormalization claim.

To prove that renormalization is an involution, we have
$$
\cS[W](q) = \sum_{n=0}^\infty \cS[W]_{n}(q) := -\sum_{n=0}^\infty
\frac{(q;q^2)_n (-1)^n q^n}{(-q^2;q^2)_n}.
$$ 
Hence $\cS[W]_{n}(q^{-1})
= -\frac{(q;q^2)_n}{(-q^2;q^2)_n}$ and $\cS[W]_{\infty}(q^{-1}) = -
\frac{(q;q^2)_\infty}{(-q^2; q^2)_\infty}$.  
To obtain the 
identity of the theorem we apply Theorem \ref{thm:AJO} by sending
$q\mapsto q^2$, then setting $t = q$ and $a = -q^2$.
\end{proof}

\begin{proof}[Proof of Theorem \ref{thm:realquadratic_quantum}]
The proof of this theorem follows from the theory of period functions
of Maass wave forms developed by Lewis and Zagier \cite{LewisZagier}
and extended by Bruggeman, Lewis, and Zagier \cite{BLZ2}.

First the valuation of the series $f_W(x)$ follows from the
definitions for the series $W(q)$ and ${\mathcal S}[W](q)$.  
The modular transformation properties follow by identifying the series
$W(q)$ and ${\mathcal S}[W](q)$ with the coefficients of the Maass
wave form, a result given in Theorem
\ref{thm:realquadratic_coefficients}.  
%
To check that the period function $h_\gamma(x)$ given in equation \eqref{eq:period_function} satisfies the desired property for all $\gamma \in \Gamma_0(4)$, it is enough to check when $\gamma$ is a generators of $\Gamma_0(4)$, i.e. 
\[
A=\begin{bmatrix} 1 & 1 \\ 0 & 1 \end{bmatrix}, 
B=\begin{bmatrix} 3 & -1 \\ 4 & -1 \end{bmatrix},
C=\begin{bmatrix} -1 & 0 \\ 0 & -1 \end{bmatrix}.
\]
When $\gamma = A$ or $C$, it is clear that $f_W\mid_{1, \gamma}(z) - \nu_1(\gamma) f_W = 0$. 
When $\gamma = B$, one could use the same calculations in Proposition \ref{prop:termbyterm} to show that the following equations holds up to a constant
\begin{equation*}
f_W(z) \stackrel{\cdot}{=} 
\left\{
\begin{aligned}
&\int_z^{i\infty} [\varphi_{0, W}(\tau), R_z(\tau)^{1/2}], \text{ when } z \in \bH,\\
 &- \int_{\overline{z}}^{i\infty} [R_z(\tau)^{1/2}, \varphi_{0, W}(\tau)], \text{ when } z \in \bH^-.
\end{aligned} 
\right.
\end{equation*}
Then the period function $h_B(z) = f_W\mid_{1, B}(z) - \nu_1(B) f_W(z)$ can be written as
$$
h_B(z) \stackrel{\cdot}{=} - \int^{i \infty}_{1/4} [\varphi_{0, W}(\tau), R_z(\tau)^{1/2}],
$$
for $z \in \bH \cup \bH^-$ with the path of integration to the left of $z$ and $\overline{z}$. 
Following the same the discussion on page 7 of \cite{zagierQuantum}, one can show that $h_B(z)$ is real-analytic on $\bR \backslash \{1/4\}$ and $C^\infty$ at $1/4$. 
%
%
\end{proof}


\section{Open Problems}\label{sec:Other}

This section contains some additional examples which do not seem to
fit as nicely into the theory of Maass waveforms, but still contain
some information about the arithmeitc if real quadratic fields.

\subsection{Other series associated with real quadratic fields}
Lovejoy \cite{lovejoy} and 
Bringmann and Kane \cite{BK} constructed a number of $q$-hypergeometric series 
associated to real quadratic fields.  These examples each realize arithmetic of the an associated
real quadratic field.  
However, they do not appear to be directly related to  Maass waveforms.  It would be interesting
to place these forms in the context of Zwegers's work \cite{zwegersIndef} which deals with 
indefinite quadratic forms.

\vspace{.1in}
The following summarizes the results of Theorems 1 and 2 of \cite{BK}. 
\begin{theorem}[Theorems 1 and 2 of \cite{BK}]\label{thm:BK_12}
With $\abs{q}<1$, 
\begin{align*}
f_1(q):= & \sum_{n=0}^\infty \frac{q^{\frac{1}{2}n(n+1)}}{(-q)_n (1-q^{2n+1})} = 
              1+ 2q + 3q^3 + q^5 + 2q^6 + 2q^7 + 4q^{10} + \cdots + 6 q^{52} + \cdots  \\
&\text{ then }  \hspace{.5in}  qf_1(q^{16}) =  \sum_{ \begin{subarray}{c} \fa \subset Z[\sqrt{2}]\\ N(\fa) \equiv 1\pmod{16} \end{subarray} } q^{N(\fa)} \\
f_2(q):= & \sum_{n=0}^\infty \frac{q^{\frac{1}{2} n(n+1)}}{(-q)_{n-1} (1-q^{2n-1})} =
   q + q^2 + 2q^3 + 2q^5 + 2q^6 + \cdots + 4q^{14} + \cdots + 3q^{77} + \cdots \\
   & \text{ then } \hspace{.5in} q^{-7} f_2(q^{16}) = \sum_{ \begin{subarray}{c} \fa \subset Z[\sqrt{2}]\\ N(\fa) \equiv 9\pmod{16} \end{subarray} } q^{N(\fa)} 
\end{align*}
\end{theorem}

Combining these identities, in the notation of Theorem \ref{thm:W2}
and \eqref{eqn:CFLZ_W1}
$$f_1(q^2) - q^{-1} f_2(q^{2}) = W_1(q) = \cS[W](q).$$
It is reasonable to expect that 
$$f_1(q^{-1}) = - \sum_{n=0}^\infty \frac{q^{2n+1}}{(-q)_n (1-q^{2n+1})}$$ 
is related to $-\sum_{\begin{subarray}{c} \fa \subset \Z[\sqrt{2}] \\ N(\fa) \equiv -1 \pmod{16} \end{subarray}} q^{N(\fa)}$
and similarly for the sum 
$f_2(q^{-1})$.   This remains a challenge. 
\begin{challenge*}
Relate $f_1(q^{-1})$ and $f_2(q^{-1})$ to 
$\sum_{\begin{subarray}{c} \fa \subset \Z[\sqrt{2}] \\ N(\fa) \equiv a \pmod{16} \end{subarray}} q^{N(\fa)}$
for an  appropriate choice of $a$. 
\end{challenge*}

The following summarizes Theorems 3 and 4 of \cite{BK}. 
\begin{theorem}[Theorems 3 and 4 of \cite{BK}]\label{thm:BK_34}
With $\abs{q}<1$, 
\begin{align*}
f_3(q):= & \sum_{n=0}^\infty \frac{(q)_{2n}q^{n} }{(-q)_{2n+1} } = 
              1- 2q^3 + q^4 + 2q^8 - 2q^{11}  + \cdots + 3 q^{24} + \cdots  \\
&\text{ then }  \hspace{.5in}  qf_3(q^{2}) =  \sum_{ \begin{subarray}{c} \fa \subset Z[\sqrt{2}] \end{subarray} } 
    \( \frac{-4}{N(\fa)} \) q^{N(\fa)} \\
f_4(q):= & \sum_{n=0}^\infty \frac{(q)_{2n+1} q^{ n+1} }{(-q)_{2n+2} } =
   q - q^2 - q^4 + 2q^7 -q^8 + q^9 -2q^{14} + \cdots + 4q^{14} + \cdots  - 3q^{98}+  \cdots \\
   & \text{ then } \hspace{.5in} f_4(q) = 
      - \sum_{  \fa \subset Z[\sqrt{2}] } (-1)^{N(\fa)}  q^{N(\fa)} 
\end{align*}
\end{theorem}

Both of $f_3(q)$ and $f_4(q)$ represent sums over all of the ideals. In view of this, it is not surprising that
$f_3(q^{-1}) =  f_3(q)$ and $f_4(q^{-1}) =- f_4(q)$.  That is, these series are (essentially) unchanged by 
$q\mapsto q^{-1}$. 

Finally, we summarize the results of Theorems 5-8 of \cite{BK}. 
\begin{theorem}[Theorems 5-8 of \cite{BK}]\label{thm:BK_5678}
\begin{align*}
f_5(q):=& \sum_{n=0}^\infty \frac{(-1)^n q^{\frac{1}{2}n(n+1)} (q)_n} {(q;q^2)_{n+1} }  
&\text{ then }& \hspace{.2in} qf_5(q^{4}) = \sum_{\begin{subarray}{c} \fa \subset \Z[\sqrt{3}] 
\\ N(\fa) \equiv 1\pmod{4} \end{subarray}}  q^{N(\fa)}, \\
f_6(q):=& \sum_{n=1}^\infty \frac{(-1)^n q^{n} (q^2;q^2)_{n-1}}{(q^n;q)_n }  
&\text{ then}&  \hspace{.2in} q^{-1}f_6(q^{4}) = - \sum_{\begin{subarray}{c} \fa \subset \Z[\sqrt{3}]\\ N(\fa) \equiv 3\pmod{4}  \end{subarray}}   q^{\abs{N(\fa)}}, \\
f_7(q):=& \sum_{n=0}^\infty \frac{ (-1)^n q^{n^2+n} (q^2;q^2)_{n} }{(-q)_{2n+1} }  
&\text{ then } & \hspace{.2in} qf_7(q^{3}) =  - \sum_{\begin{subarray}{c} \fa \subset \Z[\sqrt{3}] \\
N(\fa) \equiv 1\pmod{3} \end{subarray}} (-1)^{N(\fa)} q^{N(\fa)}, \\
f_8(q):=& \sum_{n=1}^\infty \frac{(q)_{n-1} q^{n} }{(-q^n;q)_{n} }  
& \text{ then } &\hspace{.2in} q^{-1} f_8(q^3) = \sum_{\begin{subarray}{c} \fa \subset \Z[\sqrt{3}] \\
N(\fa)\equiv 2\pmod{3}    \end{subarray}}  (-1)^{N(\fa) }q^{\abs{N(\fa)}}.
\end{align*}
\end{theorem}

Direct calculation shows that 
$$ f_5(q^{-1}) = - f_6(q) \ \ \ \text{ and } \ \ \ f_7(q^{-1}) = - f_8(q).$$
This suggests that the series $f_7(q)$ is giving us the positive Fourier coefficients and 
$f_8(q)$ is giving the negative Fourier coefficients of a Maass waveform of eigenvalue $1/4$.  
However, we have been unable to identify an appropriate Hecke character. 
A similar story is suggestive for $f_5$ and $f_6$.

\vspace{.1in}
The following series is introduced in Lovejoy \cite{lovejoy}
\begin{equation}
LL(q) := \sum_{n=1}^\infty \frac{(q)_{n-1} (-1)^n q^{\frac{1}{2} n(n+1)}}{(-q)_n}.
\end{equation}
The method of Bailey pairs gives 
\begin{equation}
LL(q) = \sum_{n=1}^\infty \sum_{-n< j \le n} (-1)^{n+j+1} q^{2n^2 - j^2}
\end{equation}
for $\abs{q} < 1$ 
(see (2.11) of \cite{lovejoy}). 
This readily yields the  following claim which appears in the proof of 
Theorem 1.1 of \cite{lovejoy}. 
\begin{theorem}[See the proof of Theorem 1.1 of \cite{lovejoy}]\label{thm:lovejoy}
For $\abs{q}<1$ we have 
$$LL(q) = \sum_{n=1}^\infty \frac{(q)_{n-1} (-1)^n q^{\frac{1}{2}n(n+1)}}{(-q)_n}  = -\sum_{ \fa \subset \Z[\sqrt{2}] : N(\fa) <0} i^{N(\fa)^2 + N(\fa)} q^{\abs{N(\fa)}}.$$
\end{theorem}

The series $LL(q)$ does not behave well under $q\mapsto q^{-1}$.  
Let 
\begin{equation}
L(q):= \sum_{n=1}^\infty \frac{q^n (q^2;q^2)_{n-1}}{(-q^2;q^2)_n}.
\end{equation} The following theorem identifies $L(q)$ and $LL(q)$. 
\begin{proposition}
We have 
$L(-q) =LL(q). $
\end{proposition}
\begin{proof}
Letting $a = -q$ and $b = q$ in \eqref{eqn:entry1.7.2} 
 we obtain 
 $$\sum_{n=0}^\infty \frac{(q)_n (-1)^n q^{\frac{1}{2}n(n+1) + n}}{(-q)_{n+1}} 
 = \sum_{n=0}^\infty \frac{(-q)^{n} (-q)_n (q)_n}{(-q^2;q^2)_{n+1}}.$$
 Multiplying both sides by $-q$ and shifting each sum to $1$ to infinity gives the result. 
\end{proof}

The series $L$ has a nice symmetry.  Namely,  
$$L(q^{-1}) = L(-q).$$ This symmetry is reflected in the arithmetic of the field $\Q[\sqrt{2}]$. 
Namely, it amounts to the fact that there is an element of norm $-1$.  

\begin{challenge*}
Relate $f_1(q), f_2(q), \ldots, f_8(q)$, and $L(q)$ and 
$f_1(q^{-1}), f_2(q^{-1}), \cdots, f_8(q^{-1})$  
to a Hecke characters or indefinite quadratic forms.
\end{challenge*}

\vspace{.1in}
Each of the series of this section it is clear that the $q$-hypergeometric series see arithmetic of the 
field in question.  However, these series are not associated with Hecke characters like the series 
$\sigma, \sigma^*$, and $W$ discussed in Sections \ref{sec:Intro} and \ref{sec:Results}.

\vspace{.1in}
Finally, returning to the series $W(q)$ of Theorems \ref{thm:W2}- \ref{thm:realquadratic_quantum}, note that   for $\abs{q}<1$
Jackson's transformation (p.14 of \cite{GR}) gives 
\begin{equation}\label{eqn:W1_jacksonTransformation}
\cS[W](q) := -\sum_{n=0}^\infty \frac{(q;q^2)_n (-1)^nq^n}{(-q^2;q^2)_n} 
   = \sum_{n=0}^\infty \frac{q^{n^2+n}}{(-q^2;q^2)_n (1+q^{2n+1})}.
\end{equation}
The series on the right  converges for $\abs{q}<1$ and $\abs{q}>1$.  It is would be interesting to relate 
the series 
\begin{align*}
\sum_{n=0}^\infty &  \frac{q^{-n^2-n}}{(-q^{-2};q^{-2})_n (1+q^{-2n-1})}
 = \sum_{n=0}^\infty \frac{q^{2n+1}}{(-q^2;q^2)_n (1+q^{2n+1})} \\
& \hspace{.8in} = q - q^2 + 2q^3 - q^4 + q^5 - 2q^6 + \cdots - 14 q^{46} + 6q^{47} + 13q^{48} - 8q^{49} + \cdots 
 \end{align*}
to $W(q)$.
\begin{challenge*}
Relate the series 
$\sum_{n=0}^\infty \frac{q^{2n+1}}{(-q^2;q^2)_n (1+q^{2n+1})}$ to $W(q)$ defined in Section \ref{sec:Results}.
\end{challenge*} 

\vspace{.2in}
\subsection{A Final Challenge}
There remains a final challenge which is the most mysterious and
interesting.  In the definition of the map $\cS[H](q)$ the series
$\cG[H](q)$ vanishes to infinite order at every root of unity where
the renormalized sum makes sense. Thus the theory of quantum modular
forms cannot see this term and does not make any predictions about its
nature.  None-the-less, there seems to be structure to these terms.
In this work, this term is often have a modular form times a divisor
function.  Can the space of ghost terms be characterized?

\end{document}